\documentclass[twoside]{amsart}
\usepackage{amsmath,amssymb,amscd,amsthm,amsxtra,amsfonts}
\usepackage{mathtools,mathrsfs,dsfont,xparse}
\usepackage{graphicx,epstopdf}
\usepackage{multirow}
\usepackage{cases}
\usepackage{enumerate}
\usepackage{xcolor}
\usepackage{tikz,pgfplots}
\usetikzlibrary{matrix}
\usepackage{mathtools}
\usepackage[margin=1 in]{geometry}
\usepackage[numbers,sort]{natbib}
\usepackage{cancel}

\mathtoolsset{showonlyrefs}

\allowdisplaybreaks[4]

\makeatletter
\newcommand{\dotDelta}{{\vphantom{\Delta}\mathpalette\d@tD@lta\relax}}
\newcommand{\d@tD@lta}[2]{%
  \ooalign{\hidewidth$\m@th#1\mkern-1mu\cdot$\hidewidth\cr$\m@th#1\Delta$\cr}%
}
\makeatother

\date{}

\usepackage{multicol,bbm}
\usepackage[colorlinks=true, pdfstartview=FitV, linkcolor=blue, citecolor=blue, urlcolor=blue]{hyperref}

\numberwithin{equation}{section}

\usepackage{cjhebrew}

\DeclareMathOperator{\re}{Re}

\usepackage[numbers,sort]{natbib}

{\theoremstyle {definition} \newtheorem {defn} {Definition} [section] }
{\theoremstyle {plain}  \newtheorem {thm} [defn] {Theorem}}
{\theoremstyle {plain}  }
{\theoremstyle {plain} \newtheorem {prop} [defn]{Proposition}}
{\theoremstyle {plain} }
{\theoremstyle {definition} \newtheorem {rmk}[defn] {Remark}}
{\theoremstyle {plain} }

\def\R{{\mathbb{R}}}
\def\C{{\mathbb{C}}}
\def\N{{\mathbb{N}}}

\newcommand{\abs}[1]{\lvert#1\rvert}
\newcommand{\norm}[1]{\left\|#1\right\|}

\newcommand{\sq}[1]{\left[#1\right]}

\begin{document}

\author[Lee and Yu]{Zachary Lee and Xueying Yu}

\address{Zachary Lee 
\newline \indent Department of Mathematics, The University of Texas at Austin \indent 
\newline \indent 
2515 Speedway, PMA 8.100, Austin, TX 78712
\newline \indent 
And
\newline \indent 
Department of Mathematics, MIT \indent 
\newline \indent 77 Massachusetts Ave, Cambridge, MA 02139\indent 
}
\email{zl9868@my.utexas.edu}

\address{Xueying  Yu
\newline \indent Department of Mathematics, Oregon State University\indent 
\newline \indent  Kidder Hall 368
Corvallis, OR 97331 \indent 
\newline \indent 
And 
\newline \indent 
Department of Mathematics, University of Washington, 
\newline \indent  C138 Padelford Hall Box 354350, Seattle, WA 98195}
\email{xueying.yu@oregonstate.edu}

\title[]{On recovering the nonlinearity for generalized higher-order Schr\"odinger equations}

\subjclass[2020]{}


\begin{abstract}
In this work, we generalize the nonlinearity-recovery result in \cite{HMG} for classical cubic nonlinear Schr\"odinger equations to  higher-order Schr\"odinger equations with a more  general nonlinearity. More precisely, we consider a spatially-localized nonlinear higher-order Schr\"odinger equation and recover the spatially-localized coefficient by the solutions with data given by small-amplitude wave packets.

\end{abstract}

\maketitle

\setcounter{tocdepth}{1}
\tableofcontents

\parindent = 10pt     
\parskip = 8pt

\section{Introduction}

In this work, we consider the following spatially localized nonlinear higher-order Schr\"odinger equations with a general nonlinearity
\begin{align}\label{eq NLS}
\begin{cases}
(i \partial_t + \frac{1}{2n} (-\Delta)^{n}) u = \beta (x) G (\abs{u}^2) u , \\
u(-T, x) = u_0 . 
\end{cases}
\end{align}
Here the solution $u : \R_t \times \R_x^d \to \C$ is a complex-valued function of time and space ($d\geq 1$), and we take  $T>0$ and $\beta \in C_c^{\infty} (\R^d)$ is nonnegative. The operator $(-\Delta)^n$ with $n \in \N_+$ is a power of the standard Laplacian $-\Delta=-\sum_{n=1}^d \partial^2_{x_n}$. In addition, the nonlinearity $G:\R_{\ge 0} \to\R_{\ge 0}$ vanishes at and is real-analytic in a neighborhood of the origin. In other words, we can write
\begin{align}
    G(x)=\sum_{k\ge 1}\frac{a_k}{k!} x^k
\end{align}
converging absolutely for $|x|<R$ for some $R>0$ that we fix throughout the paper.

The goal of this paper is to determine the nonlinear coefficient $\beta$ in \eqref{eq NLS}. Before we start the proof, let us briefly review the history. Hogan, Murphy, and Grow \cite{HMG} previously studied a similar problem for the second order case nonlinear Schr\"odinger equation (NLS) with a cubic nonlinearity ($n=1$, $G(x) =x$ in \eqref{eq NLS}). We employ a similar approach in their work, which itself is adapted from S\'a Barreto and Stefanov \cite{BS1, BS2}, who considered a similar recovery problem for the cubic wave equation (and more general nonlinear wave models).  We show that solutions to \eqref{eq NLS} with data given by small-amplitude wave packets generate phase that determines the X-ray transform of $\beta$, essentially the integrals of $\beta$ on all possible lines in its domain.

Now we state the main result of this work. 
\begin{thm}[Main result]\label{thm Main}
Let $d \geq 1$ be any dimension,  and $p>n$, $a_0, \beta\in \mathcal{S}(\mathbb{R}^d)$ with $\beta$ compactly supported. Fix $T>0$ large enough that the support of $\beta$ is contained in $\{x: |x|< T\}$. Given $\xi\in\mathbb{R}^d$ with $|\xi|=1$, define 
\begin{align}
a (t,x) = a_0 (x+ \varepsilon^{1-2n} t \xi) \,\exp\left(-i \varepsilon^{-2n} G(\varepsilon^{2p} \abs{a_0(x + \varepsilon^{1-2n} t \xi) }^2) \int_0^t \beta (\varepsilon^{-1} x + \varepsilon^{-2n} (t-s)\xi -T\xi)  \, ds\right) .
\end{align}
and
\begin{align}
v (t,x) : = \varepsilon^p a (\varepsilon^{2n} (t+T), \varepsilon (x + T\xi)) e^{ i (x \cdot \xi +  \frac{1}{2n} t)}.
\end{align}
Let $u_0(x)=v(-T,x)$. Then for $\varepsilon>0$ sufficiently small, the solution $u$ to \eqref{eq NLS} exists and satisfies  
\begin{align}\label{error}
    \norm{u - v}_{L_t^{\infty} \mathcal{F}L^1 ([-T,T] \times \R^d)} \lesssim  \max \{ \varepsilon^{p+2}, \varepsilon^{3p-2n- } \}.
\end{align}
\end{thm}
\begin{rmk}
\begin{enumerate}
\item 
We can obtain the same result by replacing the space $\mathcal{F}L^1(\R^d)$ with $H^{s}(\R^d)$, $s > d/2$, if desired, using essentially the same argument. This is because the key property we rely on for estimating solutions is that the space is an algebra, which is also true for the latter space.

\item 
The condition $p>n$ guarantees that the error between the solution and approximate solution is much smaller (if $\varepsilon>0$ is small) than both the amplitude of the former and the latter. 
\end{enumerate}
\end{rmk}

Theorem \ref{thm Main} shows that by approximately solving \eqref{eq NLS}, we can  recover the X-ray transform of  $\beta$,
\begin{align}
     X(x, \theta)=\int_{\mathbb{R}} \beta(x+t\theta)\,dt \quad \theta, x\in\R^n,  \quad |\theta|=1 .
\end{align}
In other words the X-ray transform of $\beta$ is the integral of $\beta$ on various straight lines in space,  which in turn is enough to reconstruct $\beta$ itself (see e.g. \cite{Deans}) as long as $d\ge 2$. In one dimension, we can only recover the integral of $\beta$. As explained in \cite{HMG}, we first see that the approximate solution $v(T,x)$ contains all lines integrals of $\beta$ if we let $\xi$ range over $\mathbb{S}^{d-1}$.  Additionally, in the regime $p>n$ and $\varepsilon\ll 1$, we have that the error in \eqref{error} is much smaller than the size of both the solution and the approximate solution.

\subsection{Motivation}
Part of the motivation for this work originates in \cite{BS1}, where the authors proved a similar result for the following nonlinear wave equation
\begin{align*}
    \partial_{tt} u -\Delta u +\beta |u|^2 u=0.
\end{align*}
The authors probed the equation with a wave of the form
\begin{align*}
    u(t,x)=h^{-1/2} e^{i(x\cdot\xi -t)/h} \chi(x\cdot\xi -t), \quad \text{where }0<h\ll 1
\end{align*}
outside of the support of $\beta$ and devised a geometric optics approximation similar to the one appearing in \cite{HMG} as well as in Theorem \ref{thm Main}. Our construction of an approximate solution is a modification of the ones in \cite{HMG},\cite{Carles_2010}, which are based on geometric optics solutions of NLS equations. We use the Fourier-Lebesgue space $\mathcal{F}L^1$ to establish well-posedness and stability in a similar manner. As is done in \cite{HMG}, our approximation controls the stronger norm $L^\infty_t\mathcal{F}L^1$.

Theorem \ref{thm Main} generalizes the result in \cite{HMG} (which itself establishes well-posedness and stability for a cubic NLS as well as constructs an approximate solution). See also \cite{Mur, KRMPV, Carles_2008} for related works. The result in \cite{HMG} is closely related to that done in \cite{Wa}, where the authors worked on reconstructing nonlinearities of the form $q(x)|u|^{p-1} u$ from the scattering map data. We require essentially the same smoothness on $\beta$ as in \cite{Wa} as we need to control $d/2+$ derivatives, though the compact support of $\beta$ is needed in the argument since $\varepsilon\to 0$ if we let $T\to\infty$. In this paper, we generalize the order of the NLS to any $n\ge 1$ as well as the nonlinearity to a function $G:\R_{\ge 0}\to\R_{\ge 0}$ vanishing at and real analytic in a neighborhood of the origin. We will explain the necessity of real analyticity in the next subsection.

The other part of the motivation has to do with higher-order Schr\"odinger equations, the study of which is physically interesting as such PDEs can be useful in modeling the behavior of semi-relativistic quantum particles without resorting to the spinor-valued Dirac equation like what is needed for a fully relativistic electron \cite{Carles_2012}. Indeed, for a particle of mass $m$, the non-relativistic Schr\"odinger equation with a potential $V(x)$ (the potential may also include nonlinear interactions in $u$ as is done in the equation considered in this paper) takes the following form,
\begin{align}
    -i\partial_t u= -\frac{1}{2m} \Delta u + V(x)u. 
\end{align}
The kinetic term $-\frac{1}{2m} \Delta$ (where we have set Planck's reduced constant $\hbar=1$ for convenience) represents the kinetic energy operator $|\vec{p}|^2/2m$, with $\vec{p}=-i\nabla$ the momentum opertor. This is the quantization of the usual formula for kinetic energy of a particle of mass $m$ and momentum $\vec{p}\in \mathbb{R}^d$ in classical physics: $E=|\vec{p}|^2/2m$. In relativistic classical mechanics, the kinetic energy must be modified to remain consistent with experimental observations as the speed of the particle approaches a substantial fraction of the speed of light, denoted by $c$. The modified kinetic energy takes the following form:
$$ E(\vec{p})=mc^2 \left(\sqrt{1+|\vec{p}|^2/m^2c^2}-1\right)\in [0, \infty). $$
with
$$\vec{p}=m \frac{\vec{v}}{\sqrt{1-|v|^2/c^2}}\in\R^d$$
and $\vec{v}\in\mathbb{R}^d$ the velocity vector of the particle in a specified reference frame. As long as $|p| <  mc$, meaning that $|v|<\frac{c}{\sqrt{2}}$, then $E(\vec{p})$ can be expressed as the following infinite series
$$E(\vec{p})=mc^2 \sum_{n\ge 1} (-1)^{n+1} \alpha(n) \frac{|\vec{p}|^{2n}}{(mc)^{2n}}, $$
with $\alpha(n)= \frac{1}{n} {2n-2\choose n-1} 2^{-2n+1}.$ A natural attempt to model a semi-relativistic quantum particle would be to cutoff the series for the kinetic energy of the particle at a finite number of terms $N$ and replace the scalar momentum with its quantized operator form $\vec{p}=-i\nabla$:
\begin{align}
-i\partial_t u_N &= V(x)u_N  + mc^2 \sum_{1\le n\le N} (-1)^{n+1} \alpha(n) \frac{(-\Delta)^n} {(mc)^{2n}}u_N  \\
&= V(x)u_N  -\frac{1}{2m} \Delta u_N  + mc^2 \sum_{2\le n\le N} (-1)^{n+1} \alpha(n) \frac{(-\Delta)^n} {(mc)^{2n}}u_N 
\end{align}
Notice that the first term in the series is the non-relativistic kinetic term $-\frac{1}{2m} \Delta$. The higher order terms are seen to be the relativistic corrections in powers of the dimensionless operator $\frac{|\vec{p}|^2}{m^2c^2}$ . The solutions $u_N$ converges in the $L^2$ norm to the solution $u_\infty$ to the exact semi-relativist Schr\"odinger equation  provided $\widehat{u_N}(t=0)$ is supported in $\{\vec{p}\in\R^d: |\vec{p}/mc|<1\}$ \cite{Carles_2012}, \emph{i.e.} the initial velocity of the particle has null probability of being greater than $c/\sqrt{2}$. Hence, for particles in a weakly-relativistic regime (which we will define as $|v|<c/\sqrt{2}$), such an approximation has arbitrarily small error in the wavefunction in terms of $L^2$-norms (as well as the energy) if enough terms are taken. Of course, this model is not Lorentz-invariant and neglects relativistic effects that become important at these higher energies such as creation or destruction of particles and  would not be useful to model situations involving such events. For clarity of exposition, we focus on an equation with a  unique high power of the Laplacian, but make a remark at the end of Section~\ref{Main} that explains how to adapt the analysis to the more general case where linear combinations of different powers of the Laplacian appear.

\subsection{Discussion on the Main Result}

Now we briefly describe the approach we take to prove Theorem \ref{thm Main}. The key approach is to consider an approximate solution to \eqref{eq NLS} of the following ansatz
 \begin{align}\label{ansatz}
 v (t,x) : = \varepsilon^p a (\varepsilon^{2n} (t+T), \varepsilon (x + T\xi)) \,e^{ i (x \cdot \xi +  \frac{1}{2n} t)}.
 \end{align}
After substituting this ansatz into \eqref{eq NLS}, it can be shown through a direct calculation that $v(t,x)$ approximately solves \eqref{eq NLS} (with an error term that involves derivatives of $a$ up to order $2n$) if a certain nonlinear transport equation is satisfied by $a$. To complete the proof, it is necessary to establish a theory of well-posedness and stability for \eqref{eq NLS} in the $\mathcal{F}L^1$ norm and to obtain suitable estimates for the error term in the same norm. The restriction $p>n$ is required to demonstrate that the error term is small relative to the amplitude of the solution, which is on the order of $\varepsilon^p$.

Let us remark that in \cite{HMG}, a similar ansatz as in \eqref{ansatz} was considered, which is inspired by the scaling symmetry and Galilean invariance of the linear Schr\"odinger equation. This ansatz provides a solution to the NLS with only one error term to control and yields an exact transport equation for $a$. However, when considering higher order Schr\"odinger equations of the form \eqref{eq NLS}, although the scaling symmetry is still present, Galilean invariance breaks down completely. Therefore,  additional approximate analysis is required to account for this missing symmetry and to control the many error terms arising from the ansatz form.

Also we note that the real-analyticity is needed so that we may use the closure of $\mathcal{F}L^1$ under algebraic operations to control from above the $\mathcal{F}L^1$ norm of $G(|u|^2)$ by the $\mathcal{F}L^1$ norm of $u$. In fact, using Strichartz estimates, it has become standard to establish local well-posedness for nonlinear (regular or higher-order)  Schr\"odinger equations with power-type nonlinearities, assuming initial data of optimal regularity in specific solution spaces, such as $L_t^p L_x^q$ type Strichartz spaces. However, when dealing with a general nonlinearity in our current work, we need to relax the requirement of optimal regularity and work in a more general space, such as $\mathcal{F}L^1$ or $H^s$ ($s > \frac{d}{2}$), relying on their algebraic properties to establish a local theory.
It is important to note that when working in $\mathcal{F}L^1$ or $H^s$, one must exercise caution since pointwise inequalities that hold in physical space may not necessarily hold in Fourier space. Therefore, properties like Lipschitz continuity (which are defined in a pointwise manner in physical space) are not necessarily preserved under the $\mathcal{F}L^1$ or $H^s$ norms. In other words, $|f| \leq C |g|$ for $C>0$ neither implies $\norm{f}_{\mathcal{F}L^1} \leq C '\norm{g}_{\mathcal{F}L^1}$ nor $\norm{f}_{H^s} \leq C' \norm{g}_{H^s}$ for some $C'>0$. 
This occurs because even though $f$ may be pointwise dominated by $g$ in space, it may still be arbitrarily rougher than $g$, hence have arbitrarily large $\mathcal{F}L^1$ norms. This is precisely why we impose the condition of real-analyticity on the nonlinearity, as it provides a certain amount of structure that enables us to handle the nonlinear estimates effectively.

We organize the rest of the paper as such: In Section \ref{prelim} we give some notation and collect some various estimates.  In Section \ref{WPS}, we prove well-posedness and stability of \eqref{eq NLS} in the space $\mathcal{F} L^1$.  In Section \ref{Approx}, we construct the approximate solution $v$ to \eqref{eq NLS} appearing in Theorem \ref{thm Main} and prove estimates for the error term involving $a$ and its derivatives.  Finally, in Section \ref{Main}, we carry out the proof of Theorem \ref{thm Main}.

\subsection*{Acknowledgements}
Both authors would like to thank Jason Murphy and Yang Zhang for very insightful conversations and comments on a preliminary draft of this paper. The authors are very grateful to the anonymous referees for their valuable comments and suggestions. Z. L. was supported by the Undergraduate Research Opportunities Program at the Massachusetts Institute of Technology. X.Y. was partially supported by an AMS-Simons travel grant.

\section{Preliminaries}\label{prelim}

In this section, we define the function spaces that will be used in the rest of this paper.

\subsection{Notations}
We use the usual notation that $A \lesssim  B$ to denote an estimate of the form $A \leq C B$, for some constant $0 < C < \infty$ depending only on the {\it a priori} fixed constants of the problem. We also use $a+$ and $a-$ to denote expressions of the form $a + \sigma$ and $a - \sigma$, for any $0 <  \sigma \ll 1$.

\subsection{Fourier Transforms and Function Spaces}
\begin{defn}[Fourier transform]
Let $\widehat{u} $ or $\mathcal{F} u$ be the Fourier transform of $u$ defined as follows
\begin{align}
\widehat{u} (y) = \mathcal{F} u (y)= \frac{1}{(2\pi)^{d/2}} \int_{\R^d} e^{-i x \cdot y} u(x) \, dx ,
\end{align}
and $\mathcal{F}^{-1} u$ be the inverse Fourier transform
\begin{align}
\mathcal{F}^{-1} v (x) = \frac{1}{(2\pi)^{d/2}} \int_{\R^d} e^{i x \cdot y} v(y) \, dy .
\end{align}
\end{defn}

\begin{defn}[Fourier-Lebesgue space]\label{defn FL}
We recall the Fourier-Lebesgue space $\mathcal{F}L^1$ equipped with the norm
\begin{align}
\norm{u}_{\mathcal{F}L^1 (\R^d)} : = \norm{\widehat{u}}_{L^1  (\R^d)} .
\end{align}
\end{defn}

\begin{prop}[Embedding property]
Using the Hausdorff–Young inequality and Definition \ref{defn FL}, we have
\begin{align}
\norm{u}_{L^{\infty}  (\R^d)} = \norm{\mathcal{F}^{-1} \widehat{u}}_{L^{\infty} (\R^d)} \lesssim \norm{\widehat{u}}_{L^1 (\R^d)} = \norm{u}_{\mathcal{F}L^1 (\R^d)}.
\end{align}

By Cauchy–Schwarz inequality, the integrability of $(1 + \abs{y}^2)^{-\frac{d}{4}-}$ in $L^2 (\R^d)$ and Plancherel theorem, we obtain
\begin{align}\label{FL1}
\norm{u}_{\mathcal{F}L^1 (\R^d)} \lesssim \norm{(1 + \abs{y}^2)^{\frac{d}{4}+} \widehat{u}}_{L^2 (\R^d)} \lesssim \norm{u}_{H^{\frac{d}{2}+} (\R^d)}.
\end{align}
\end{prop}

\begin{prop}[Algebra property]\label{prop Algebra}
Using Definition \ref{defn FL} and Young's inequality, we write 
\begin{align}
\norm{u v}_{\mathcal{F}L^1 (\R^d)} = \norm{\mathcal{F} [uv]}_{L^1 (\R^d)} = \norm{\widehat{u} * \widehat{v}}_{L^1 (\R^d)} \leq \norm{\widehat{u}}_{L^1 (\R^d)} \norm{\widehat{v}}_{L^1 (\R^d)} = \norm{u }_{\mathcal{F}L^1 (\R^d)} \norm{v}_{\mathcal{F}L^1 (\R^d)}  .
\end{align}
Hence we conclude that $\mathcal{F} L^1$ is an algebra.
\end{prop}

\section{Well-posedness and Stability in $\mathcal{F}L^1$}\label{WPS}
In this section, we present a local well-posedness argument (Proposition \ref{prop WP}) and a stability result (Proposition \ref{prop Stab}) (see also \cite{Carles_2010}).
\begin{prop}[Well-posedness in $\mathcal{F}L^1$]\label{prop WP}

Let $d$ be any dimension, $ \beta  \in \mathcal{F}L^1 (\R^d) $, and $T>0$. There exists $\delta_0 = \delta_0 (T , \norm{\beta}_{\mathcal{F}L^1 (\R^d)}, G) >0$ such that for any $u_0 \in \mathcal{F}L^1 (\R^d)$ with
\begin{align}
\norm{u_0}_{\mathcal{F}L^1  (\R^d)} < \delta_0 , 
\end{align}
there exists a unique solution $u \in L_t^{\infty} ([-T, T] ; \mathcal{F}L^1(\R^d) )$ to \eqref{eq NLS} satisfying
\begin{align}\label{WP1} 
\norm{u}_{L_t^{\infty} \mathcal{F}L^1 ([-T, T]  \times \R^d)} \leq 2 \norm{u_0}_{\mathcal{F}L^1  (\R^d)}
\end{align}

\end{prop}

\begin{proof}[Proof of Proposition \ref{prop WP}]
We will establish a solution to the Duhamel formula for \eqref{eq NLS}, 
\begin{align}
    u(t)=e^{i(t+T) (-\Delta)^{n}/2n}u_0-i\int_{-T}^t e^{i(t-s) (-\Delta)^{n}/2n}  \beta (x)[G(|u|^2) u](s)\,ds.
\end{align}

Fix $T>0$ and $\beta \in \mathcal{F}L^1$. Let $\delta_0>0$ be determined later and let $u_0\in\mathcal{F}L^1$ satisfy
\begin{align}
    \norm{u_0}_{\mathcal{F}L^1  (\R^d)}<\delta_0.
\end{align}
Define the complete metric space $(B,d)$ by
\begin{align}
B=\{u\in L^\infty([-T,T]; \mathcal{F}L^1(\mathbb{R}^d)): \norm{u}_{L^\infty_t \mathcal{F}L^1([-T,T] \times \mathbb{R}^d)} \le 2\norm{u_0}_{\mathcal{F}L^1 (\R^d)}\}
\end{align}
and 
\begin{align}
    d(u,v)= \norm{u -v}_{L^\infty_t \mathcal{F}L^1([-T,T] \times \mathbb{R}^d)}.
\end{align}
We also define
\begin{align}
    \Psi[u](t):=e^{i(t+T) (-\Delta)^{n}/2n}u_0-i\int_{-T}^t e^{i(t-s) (-\Delta)^{n}/2n}[ \beta (x) G(|u(s)|^2) u(s)]\,ds.
\end{align}
We will show that for small enough $\delta_0$, $\Psi$ is a contraction on $B$. Let us first note that $x\mapsto G(x^2)$ is real-analytic with radius of convergence $R^{1/2}>0$. We notice that we also have $\norm{u}_{L_t^\infty L_x^\infty} < 2\delta_0$. Hence, we have that for $\delta_0 < \frac{R^{1/2}}{2}$,
\begin{align}
    G(|u|^2)(x)=\sum_{k\ge 1} \frac{a_k}{k!}|u|^{2k}(x)\quad \text{for any } x\in\R^d. 
\end{align}
Hence, taking the $\mathcal{F}L^1$ norm and utilizing the triangle inequality and the algebra property, we have for all $u\in B$, 
\begin{align}
    \norm{G(|u|^2)}_{\mathcal{F}L^1} &\le \sum_{k\ge 1} \frac{\abs{a_k}}{k!} \norm{u}_{\mathcal{F}L^1}^{2k} \le \sum_{k\ge 1} \frac{\abs{a_k}}{k!} 2^{2k} \delta_0^{2k} \\
    &=\delta_0^2 \sum_{k\ge 1} \frac{\abs{a_k}}{k!} 2^{2k} \delta_0^{2k-2} \le  \delta_0^2 \sum_{k\ge 1} \frac{\abs{a_k}}{k!} 2^{2k} \left(\frac{R^{1/2}}{4}\right)^{2k-2} \\ 
    &=\delta_0^2 \frac{16}{R} \sum_{k\ge 1} \frac{\abs{a_k}}{k!} 2^{-2k} R^{k} =\delta_0^2 \frac{16}{R} \sum_{k\ge 1} \frac{\abs{a_k}}{k!}  \left(\frac{R}{4}\right)^k
\end{align}

provided we take $\delta_0 <  \frac{R^{1/2}}{4}$ with $R>0$ the radius of real-analyticity of $G$.

Now, let $u\in B$. For each $t\in[-T,T]$, we use $e^{it (-\Delta)^n/2n}=\mathcal{F}^{-1}e^{it|y|^{2n}/2n}\mathcal{F}$, the triangle inequality, $\norm{u}_{L_t^\infty \mathcal{F}L^1} \le 2 \norm{u_0}_{\mathcal{F}L^1}<2\delta_0$ and the algebra property (Proposition \ref{prop Algebra}) to estimate
\begin{align}
\norm{\mathcal{F}[\Psi u](t)}_{L^1  (\R^d)} &\le \norm{\widehat{u}_0}_{L^1 (\R^d)} + \norm{\int_{-T}^t e^{i(t-s)|y|^{2n}/{2n}}\mathcal{F}[ \beta (x) G(|u(s)|^2) u(s)]\,ds}_{L^1 (\R^d) } \\ 
&\le \norm{u_0}_{\mathcal{F}L^1 (\R^d)} + \int_{-T}^t \norm{\beta}_{\mathcal{F}L^1 (\R^d)}\norm{G(|u(s)|^2)}_{\mathcal{F}L^1 (\R^d)}\norm{u(s)}_{\mathcal{F}L^1 (\R^d)} \, ds \\
&\le \norm{u_0}_{\mathcal{F}L^1 (\R^d)} + 32\delta_0^2\,T  \,\frac{1}{R} \, \norm{\beta}_{\mathcal{F}L^1 (\R^d)}\norm{u_0}_{\mathcal{F}L^1 (\R^d)}\sum_{k\ge 1} \frac{\abs{a_k}}{k!}  \left(\frac{R}{4}\right)^k \\
&\le 2 \norm{u_0}_{\mathcal{F}L^1 (\R^d)}
\end{align}
provided
$$\delta_0\le \left(\frac{R}{32T\,\sum_{k\ge 1} \frac{\abs{a_k}}{k!}  \left(\frac{R}{4}\right)^k\,\norm{\beta}_{\mathcal{F}L^1 (\R^d)}}\right)^{1/2}.$$
Taking the supremum over $t\in[-T, T]$, we see that $\Psi$ maps $B$ to itself.
Next, we let $u, v \in B$ and write
\begin{align}
    u\,G(|u|^2)-v\,G(|v^2|)=\sum_{k\ge 1}\frac{a_k}{k!}\left(u\,|u|^{2k}-v\,|v|^{2k}\right).
\end{align}
Now notice that the expression $u\,|u|^{2k}-v\,|v|^{2k}$ can be written as 
\begin{align}
    u\,|u|^{2k}-v\,|v|^{2k} &= (u-v)|u|^{2k} + v\left(|u|^{2k}-|v|^{2k}\right) \\ 
    &= (u-v)|u|^{2k} + v\left(|u|^2-|v|^2\right)\sum_{j=1}^k |u|^{2k-2j} |v|^{2j-2} \\
    &= (u-v)|u|^{2k} + v\left[(u-v)\overline{u} +v(\overline{u-v})\right]\sum_{j=1}^k |u|^{2k-2j} |v|^{2j-2}.
\end{align}

Hence, taking the $\mathcal{F}L^1$ norm and utilizing the algebraic property of the Fourier-Lebesgue norm, for  $\norm{u}_{\mathcal{F}L^1}, \norm{v}_{\mathcal{F}L^1} <2\delta_0$,
\begin{align}
    \norm{u\,|u|^{2k}-v\,|v|^{2k}}_{\mathcal{F}L^1(\R^d)} &\le \norm{u-v}_{\mathcal{F}L^1(\R^d)} 2^{2k}\left(\delta_0^{2k}+2\delta_0^2\sum_{j=1}^k \delta_0^{2k-2} \right) \\ 
    &= \norm{u-v}_{\mathcal{F}L^1(\R^d)}2^{2k}\delta_0^{2k} \left(1+2k\right)
\end{align}

Hence, utilizing the analyticity of $G$ near the origin
\begin{align}
    \bigg\Vert u\,G(|u|^2)-v\,G(|v^2|)\bigg\Vert_{\mathcal{F} L^1 (\R^d)} & \le \sum_{k\ge 1}\frac{|a_k|}{k!} (2k+1) 2^{2k} \delta_0^{2k}\big\Vert u-v\big\Vert_{\mathcal{F}L^1 (\R^d)} \\
    &\le  \delta_0^2 \big\Vert u-v\big\Vert_{\mathcal{F}L^1 (\R^d)} \sum_{k\ge 1}\frac{(2k+1)\,|a_k|}{k!} 2^{2k} \delta_0^{2k-2} \\
    &\le  \delta_0^2 \big\Vert u-v\big\Vert_{\mathcal{F}L^1 (\R^d)} \sum_{k\ge 1}\frac{(2k+1)\,|a_k|}{k!} 2^{2k} \left(\frac{R^{1/2}}{4}\right)^{2k-2} \\
     &\le \frac{16}{R} \delta_0^2 \big\Vert u-v\big\Vert_{\mathcal{F}L^1 (\R^d)} \sum_{k\ge 1} \frac{(2k+1)\,|a_k|}{k!}  \left(\frac{R}{4}\right)^k 
\end{align}
provided we take $\delta_0 \le \frac{R^{1/2}}{4}$, where $R$ is such that $G(x)$ is real-analytic for $|x|<R$,   despite the factor of $2k+1$ because of the fact that $\lim_{k\to\infty} |2k+1|^{\frac{1}{k}}=1$ applied to the standard \emph{Cauchy root test}.

Hence, we have
\begin{align}
    \norm{\mathcal{F}[\Psi u](t) - \mathcal{F}[\Psi v](t)}_{L^1 (\R^d)} &\le \int_{-T}^t \norm{\beta}_{\mathcal{F}L^1 (\R^d)} \norm{G(|u|^2) u - G(|v|^2) v}_{\mathcal{F}L^1 (\R^d)}\,ds \\
    &\le  \frac{32T}{R} \delta_0^2 \big\Vert u-v\big\Vert_{\mathcal{F}L^1 (\R^d)} \sum_{k\ge 1} \frac{(2k+1)\,|a_k|}{k!}  \left(\frac{R}{4}\right)^k\norm{\beta}_{\mathcal{F}L^1 (\R^d) } \\
    &\le \frac{1}{2}\norm{u-v}_{L^\infty_t \mathcal{F}L^1([-T,T] \times \mathbb{R}^d)}
\end{align}
provided
$$\delta_0\le \left(\frac{R}{64T\norm{\beta}_{\mathcal{F}L^1}\sum_{k\ge 1} \frac{(2k+1)\,|a_k|}{k!}  \left(\frac{R}{4}\right)^k}\right)^{1/2}.$$

 Therefore, to justify all our arguments, we need to take 
$$\delta_0 < \frac{ R^{1/2}}{4}\min \left(1, \frac{1}{2\sqrt{T\norm{\beta}_{\mathcal{F}L^1}\sum_{k\ge 1} \frac{(2k+1)\,|a_k|}{k!}  \left(\frac{R}{4}\right)^k}}\right).$$

Taking the supremum over $t$, we see that $\Psi$ is a contraction mapping, and thus by the Banach fixed point theorem, has a unique fixed point $u\in B$, yielding the desired solution to \eqref{eq NLS}. 
\end{proof}

\begin{prop}[Stability in $\mathcal{F}L^1$]\label{prop Stab}

Let $d$ be any dimension, $\beta \in \mathcal{F}L^1 (\R^d)$, and $T>0$. Let $\delta_0 >0$ be as in Proposition \ref{prop WP}.

If $v : [-T, T] \times \R^d \to \C$ satisfies
\begin{align}
\norm{v (-T,x)}_{\mathcal{F}L^1 (\R^d)} < \delta_0  \label{stability1}, \\
\norm{v}_{L_t^{\infty} \mathcal{F}L^1 ([-T,T] \times \R^d)} \lesssim \delta_0 \label{stability2}, 
\end{align}
and
\begin{align}
\norm{\int_{-T}^t e^{i(t-s) (-\Delta)^{n}/2n} [(i \partial_t + \frac{1}{2n}(-\Delta)^n)v -  \beta (x) G(\abs{v}^2) v] (s) \, ds}_{L_t^{\infty}\mathcal{F}L^1 ([-T,T] \times \R^d)} = \delta \label{eq2},
\end{align}
then the solution $u$ to \eqref{eq NLS} with $u(-T,x) = v(-T ,x) $ exists on $[-T,T]$ and satisfies 
\begin{align}
\norm{u - v}_{L_t^{\infty} \mathcal{F}L^1 ([-T,T] \times \R^d)} \lesssim \delta \label{stability3}. 
\end{align}
\end{prop}

\begin{proof}[Proof of Proposition \ref{prop Stab}]

We fix $T>0$ and $\beta \in \mathcal{F} L^{1}(\R^d)$. Assume that $v:[-T, T] \times \mathbb{R}^{d} \rightarrow \mathbb{C}$ satisfies \eqref{stability1}, \eqref{stability2}.  We let $u:[-T, T] \times \mathbb{R}^{d} \rightarrow \mathbb{C}$ be the solution to \eqref{eq NLS} with $\left.u\right|_{t=-T}=\left.v\right|_{t=-T}$ satisfying \eqref{WP1}, whose existence is guaranteed by Proposition \ref{prop WP} and \eqref{stability1}.

We now observe that the difference $u-v$ satisfies the Duhamel formula
\begin{align}
u(t)-v(t)=-i \int_{-T}^{t} e^{i(t-s) (-\Delta)^{n} / 2n}\left(\beta(x)\left[G(|u|^{2}) u-G(|v|^{2}) v\right](s)-E(s)\right) d s,
\end{align}
where
\begin{align}
E(t):=\left(i \partial_{t}+\frac{1}{2n} (-\Delta)^{n} \right) v- \beta(x) G(|v|^{2}) v .
\end{align}

We now estimate as we did in the proof of Proposition \ref{prop WP}. Using \eqref{eq NLS}, \eqref{eq2}, \eqref{stability2}, and \eqref{WP1}, we have
\begin{align}
\norm{u(t) - v(t)}_{\mathcal{F}L^1 (\R^d)} 
& \le \frac{32T}{R} \delta_0^2 \big\Vert u-v\big\Vert_{\mathcal{F}L^1 (\R^d)} \sum_{k\ge 1} \frac{(2k+1)\,|a_k|}{k!}  \left(\frac{R}{4}\right)^k\norm{\beta}_{\mathcal{F}L^1 (\R^d) }+\delta.
\end{align}
Taking the supremum over $t \in[-T, T]$ and remembering that $\delta_{0}$ satisfies  
$$\delta_0\le \left(\frac{R}{64T\norm{\beta}_{\mathcal{F}L^1}\sum_{k\ge 1} \frac{(2k+1)\,|a_k|}{k!}  \left(\frac{R}{4}\right)^k}\right)^{1/2},$$
we deduce that
\begin{align}
\norm{u-v}_{L_{t}^{\infty} \mathcal{F}L^{1}\left([-T, T] \times \mathbb{R}^{d}\right)} \leq \frac{1}{2} \norm{u-v}_{L_{t}^{\infty} \mathcal{F}L^{1}\left([-T, T] \times \mathbb{R}^{d}\right)}+\delta,
\end{align}
which yields \eqref{stability3}, as desired.

\end{proof}

\section{Approximation Solutions}\label{Approx}
In this section, we consider an approximate solution based on work done in \cite{Carles_2010} to \eqref{eq NLS} and show the estimates needed to apply the stability result (Proposition \ref{prop Stab}).
\begin{prop}\label{prop Appx}
Let $d$ be any dimension. Fix $T>0$, $0 < \varepsilon \ll 1$, and $\xi \in \R^d$ with $\abs{\xi} =1$. Let $a_0 \in \mathcal{F}L^1(\mathbb{R}^d)$ and define
\begin{align}
a (t,x) = a_0 (x+ \varepsilon^{1-2n} t \xi) \,\exp\left(-i \varepsilon^{-2n} G(\varepsilon^{2p} \abs{a_0(x + \varepsilon^{1-2n} t \xi) }^2) \int_0^t \beta (\varepsilon^{-1} x + \varepsilon^{-2n} (t-s)\xi -T\xi)  \, ds\right) .
\end{align}
and
\begin{align}
v (t,x) : = \varepsilon^p a (\varepsilon^{2n} (t+T), \varepsilon (x + T\xi)) \,e^{ i (x \cdot \xi +  \frac{1}{2n} t)} .
\end{align}
Then we have the following 
\begin{enumerate}
\item 
The function $v$ satisfies the initial condition
\begin{align}
v(-T, x) = u_0 (x) := \varepsilon^p a_0 (\varepsilon (x+ T)) \,e^{ i (x \cdot \xi +  \frac{1}{2n} t)} ;
\end{align}

\item
We have the following identity
\begin{align}
(i \partial_t + \frac{1}{2n} (-\Delta)^n) v -  \beta (x) G(\abs{v}^2) v & = \varepsilon^{p+2n} e^{i \Phi} \sq{ i \partial_t a - i \varepsilon^{1-2n} \xi \cdot \nabla a  -  \varepsilon^{-2n}  \beta(x) G(\varepsilon^{2p} \abs{a}^2) a } \\
& \quad + \varepsilon^{p+2n} e^{i \Phi} \sq{\sum_{j=2}^{2n} \mathcal{O} (\varepsilon^{j-2n}  e^{i\Phi} D_x^j a ) } ,
\end{align}
where $D_x^j = \sum\limits_{ \substack{ (j_1, j_2 , \ldots , j_d) \\ j_1 + j_2 + \cdots + j_d = j}} \partial_{x_1}^{j_1}  \partial_{x_2}^{j_2}  \cdots \partial_{x_d}^{j_d} $;

\item
The function $a$ satisfies  the following estimates
\begin{align}
\norm{a}_{\mathcal{F}L^1} &  \lesssim \norm{a}_{H_x^{\frac{d}{2}+} (\R^d)}   \lesssim  \max \{ 1 , \varepsilon^{2p - 2n - }\} ,\\
\norm{D_x^j a}_{\mathcal{F}L^1}  &  \lesssim  \max \{ 1 , \varepsilon^{2p - 2n -j - }\} .
\end{align}
\end{enumerate}
\end{prop}

\begin{proof}[Proof of Proposition \ref{prop Appx}]

Let 
\begin{align}
\Phi = \Phi (t,x) = x \cdot \xi +  \frac{1}{2n} t , 
\end{align}
and consider an approximate solution of the following form 
\begin{align}
v (t,x) = \varepsilon^p a (\varepsilon^{2n} (t+T) , \varepsilon (x +T\xi)) e^{i\Phi} . 
\end{align}

Then we compute
\begin{align}
\partial_t v & =  \varepsilon^p e^{i\Phi} a (\varepsilon^{2n} (t+T) , \varepsilon (x +T\xi))  (\frac{i}{2n}) +  \varepsilon^p e^{i\Phi} \partial_t a (\varepsilon^{2n} (t+T) , \varepsilon (x +T\xi)) \varepsilon^{2n} \\
\partial_{x_i} v & = \varepsilon^p e^{i\Phi} a (\varepsilon^{2n} (t+T) , \varepsilon (x +T\xi)) (i \xi_i) + \varepsilon^p  e^{i\Phi} \partial_{x_i} a (\varepsilon^{2n} (t+T) , \varepsilon (x +T\xi)) \varepsilon \\
\partial_{x_i}^2 v & = \varepsilon^p e^{i\Phi} a (\varepsilon^{2n} (t+T) , \varepsilon (x +T\xi)) (i \xi_i)^2 + 2 \varepsilon^p e^{i\Phi} \partial_{x_i} a (\varepsilon^{2n} (t+T) , \varepsilon (x +T\xi)) (i \xi_i) \varepsilon \\
& \quad + \varepsilon^p e^{i\Phi} \partial_{x_i}^2 a (\varepsilon^{2n} (t+T) , \varepsilon (x +T\xi)) \varepsilon^2
\end{align}
and 
\begin{align}
(i \partial_t + \frac{1}{2n} (-\Delta)^{n} ) v & = \cancel{-\frac{1}{2n} \varepsilon^p e^{i\Phi} a (\varepsilon^{2n} (t+T) , \varepsilon (x +T\xi)) }  +  i \varepsilon^{p+2n} e^{i\Phi} \partial_t a (\varepsilon^{2n} (t+T) , \varepsilon (x +T\xi)) \\
& \quad + \cancel{\frac{1}{2n} \varepsilon^p e^{i\Phi} a (\varepsilon^{2n} (t+T) , \varepsilon (x +T\xi)) } \\
& \quad - i \varepsilon^{p+1}  e^{i\Phi} \xi \cdot \nabla a (\varepsilon^{2n} (t+T) , \varepsilon (x +T\xi))\\
& \quad + \sum_{j=2}^{2n} \mathcal{O} (\varepsilon^{p+j}  e^{i\Phi} D_x^j a ) \\
& =   i \varepsilon^{p+2n} e^{i\Phi} \partial_t a (\varepsilon^{2n} (t+T) , \varepsilon (x +T\xi)) - i \varepsilon^{p+1}  e^{i\Phi} \xi \cdot \nabla a (\varepsilon^{2n} (t+T) , \varepsilon (x +T\xi))\\
& \quad + \sum_{j=2}^{2n} \mathcal{O} (\varepsilon^{p+j}  e^{i\Phi} D_x^j a ) .
\end{align}

Combining
\begin{align}
G(\abs{v}^2) v = \varepsilon^p  e^{i\Phi} G(\varepsilon^{2p} \abs{a}^2) a ,
\end{align}
we have
\begin{align}
(i \partial_t + \frac{1}{2n} (-\Delta)^n) v -  \beta (x) G(\abs{v}^2) v & = \varepsilon^{p+2n} e^{i \Phi} \sq{ i \partial_t a - i \varepsilon^{1-2n} \xi \cdot \nabla a  -  \varepsilon^{-2n}  \beta(x) G(\varepsilon^{2p} \abs{a}^2) a } \label{eq Trans}\\
& \quad + \varepsilon^{p+2n} e^{i \Phi} \sq{\sum_{j=2}^{2n} \mathcal{O} (\varepsilon^{j-2n}  e^{i\Phi} D_x^j a )  }  , \label{eq Err}
\end{align}
where $a$ and its derivatives are evaluated at $(t' , x') = (\varepsilon^{2n} (t+T) , \varepsilon (x + T\xi))$.

Now we wish to find $a$ such that $a$ solves,
\begin{align}
\begin{cases}
i \partial_t a - i \varepsilon^{1-2n} \xi \cdot \nabla a  - \varepsilon^{-2n}   \beta (\varepsilon^{-1} x - T\xi) G(\varepsilon^{2p}\abs{a}^2) a  =0 ,\\
a(0,x) = a_0 ,
\end{cases}
\end{align}
which is the first term in \eqref{eq Trans}, and  as a result, we can treat \eqref{eq Err} as an error term during this approximation process.

Using the method of characteristics, we set  $x(t) = x_0 - \varepsilon^{1-2n} t\xi$ for some $x_0 \in \R^d$ and $b(t) = a (t, x(t))$. 
We reduce the equation along the characteristic in the following form
\begin{align}\label{eq 2}
\frac{d}{dt} b =  -i \varepsilon^{-2n}  \beta(\varepsilon^{-1} x(t) - T\xi) G(\varepsilon^{2p} \abs{b}^2) b .
\end{align}
Notice that 
\begin{align}
\frac{d}{dt} \abs{b}^2 = 2\re [ - i \varepsilon^{-2n}  \beta(\varepsilon^{-1} x(t) - T\xi) G(\varepsilon^{2p} \abs{b}^2) \abs{b}^2 ] = 0 ,
\end{align}
hence $\abs{b}^2 = \abs{a_0(x_0)}^2 $ is a constant.

Then \eqref{eq 2} becomes
\begin{align}
\frac{d}{dt} b & = - i \varepsilon^{-2n}  \beta(\varepsilon^{-1}x (t) - T\xi) G(\varepsilon^{2p} \abs{a_0(x_0)}^2) b ,
\end{align}
then
\begin{align}
b & = a_0 (x_0) \,\exp\left( - i \varepsilon^{-2n} G(\varepsilon^{2p} \abs{a_0(x_0)}^2) \int_0^t \beta(\varepsilon^{-1} x(s) - T\xi) \, ds\right) .
\end{align}
Hence
\begin{align}
a (t,x) = a_0 (x+ \varepsilon^{1-2n} t \xi) \,\exp\left(-i \varepsilon^{-2n} G(\varepsilon^{2p} \abs{a_0(x + \varepsilon^{1-2n} t \xi) }^2) \int_0^t \beta (\varepsilon^{-1} x + \varepsilon^{-2n} (t-s)\xi -T\xi)  \, dx\right) .
\end{align}

Next, we will compute  $H^N, N \sim d/2$ norms of $a$  to control from above the $\mathcal{F}L^1$ norm of $a$ and its derivatives using the estimate $\norm{a}_{\mathcal{F}L^1(\mathbb{R}^d)} \lesssim \norm{a}_{H^{d/2+}(\mathbb{R}^d)}$. When computing the derivative of $a$, we have the following two cases
\begin{enumerate}
\item
if $\nabla$ falls on $a_0$, this is a good case, which gives a constant multiple of $a$
\item
if $\nabla$ falls on  in the exponential term, there are two subcases
\begin{enumerate}
\item 
$\nabla$ hits $G$ in the exponential term, and it contributes $\varepsilon^{-2n}$ from the phase and $\varepsilon^{2p}$ from the chain rule in $G$. 
\item
$\nabla$ hits the integral in the exponential term, and it contributes $\varepsilon^{-2n}$ from the phase and $\varepsilon^{-1}$ from $\beta$, also due to  $\beta$ being localized in a ball of radius $\varepsilon T$, we have an $\varepsilon^{\frac{d}{2}}$. Notice that we have the following property of $G$
\begin{align}
\abs{G (\varepsilon^{2p} x)} \lesssim  \varepsilon^{2p} |x|
\end{align}
for small enough $\varepsilon$ and bounded $x$. 
hence we have an $\varepsilon^{2p}$ from $G$. 
\end{enumerate}
\end{enumerate}
Hence 
\begin{align}
\norm{a}_{H_x^1} \lesssim \max \{ 1 , \varepsilon^{-2n +2p} , \varepsilon^{-2n + 2p - 1 + \frac{d}{2}} \}  .
\end{align}
These terms are acceptable, since we can ask $p >n$. In this case, $\varepsilon^{2p-2n}\ll 1$, and so we neglect that term in the rest of the estimates.

Moreover, we can compute higher Sobolev norms of $a$ following the same idea,
\begin{align}
\norm{a}_{H_x^1} & \lesssim  \max \{ 1  , \varepsilon^{-2n + 2p - 1 + \frac{d}{2}} \} , \\
\norm{a}_{H_x^2} & \lesssim \max \{ 1 , \varepsilon^{-2n + 2p - 2 + \frac{d}{2}} \}  , \\
 \norm{a}_{H_x^N} & \lesssim \max \{ 1 , \varepsilon^{ -2n  + 2p -N + \frac{d}{2}}\} .
\end{align}
In particular, we have
\begin{align}
\norm{a}_{H_x^{\frac{d}{2}+}} \lesssim \max \{ 1, \varepsilon^{- 2n +2p + \frac{d}{2} - \frac{d}{2}  -}\} =  \max \{ 1 , \varepsilon^{2p - 2n  -}\} ,
\end{align}
which implies
\begin{align}
\norm{a}_{\mathcal{F}L^1} & \lesssim \norm{a}_{H_x^{\frac{d}{2}+}} \lesssim  \max \{ 1 , \varepsilon^{2p - 2n - }\} , \\
\norm{D_x^j a}_{\mathcal{F}L^1}  &  \lesssim  \max \{ 1 , \varepsilon^{2p - 2n -j - }\}  .
\end{align}
\end{proof}

\section{Proof of Main Theorem and Further remarks}\label{Main}
In this section, we prove the main result in this note Theorem \ref{thm Main} by combining the well-posedness and stability result and the approximation theorem.
\begin{proof}[Proof of Theorem \ref{thm Main}]
We define the function $v(t,x)$ satisfying the initial condition
\begin{align}
v(-T, x) = u_0 (x) := \varepsilon^p a_0 (\varepsilon (x+ T)) \,e^{ i (x \cdot \xi +  \frac{1}{2n} t)}.
\end{align}
given in Proposition \ref{prop Appx}.

We first observe that
\begin{align}
    \norm{v}_{L_t^{\infty} \mathcal{F}L^1} \lesssim \varepsilon^p \norm{a}_{L_t^{\infty} \mathcal{F}L^1}\lesssim \varepsilon^p 
\end{align}
so that \eqref{stability2} is satisfied for sufficiently small $\varepsilon$. Next, we use the estimates in Proposition \ref{prop Appx} for the norm $L_t^{\infty} \mathcal{F}L^1$ of $a$ and its derivatives to estimate
\begin{align}
    & \quad \norm{\int_{-T}^t e^{i(t-s) (-\Delta)^n/2n} [(i \partial_t + \frac{1}{2n}(-\Delta)^n)v -  \beta (x) G(\abs{v}^2) v] (s) \, ds}_{L_t^{\infty}\mathcal{F}L^1 ([-T,T] \times \R^d)} \\
    & \lesssim  \varepsilon^{p+2n} \sum_{j=2}^{2n} \varepsilon^{j-2n} \norm{D_x^j a}_{L_t^{\infty} \mathcal{F}L^1} \\
    & \lesssim \sum_{j=2}^{2n} \varepsilon^{p+2n} \max\{\varepsilon^{j-2n} , \varepsilon^{j-2n} \varepsilon^{2p-2n-j-} \} \\
    & \leq \max \{ \varepsilon^{p+2}, \varepsilon^{3p-2n- } \}.
\end{align}
where we need $3p -2n - \geq p$, hence $p>n$.

Thus, \eqref{eq2} is satisfied for sufficiently small $\varepsilon$ and  $p>n$. Using Proposition \ref{prop Stab}, we deduce the existence of $u$ such that
\begin{align}
    \norm{u - v}_{L_t^{\infty} \mathcal{F}L^1 ([-T,T] \times \R^d)} \lesssim  \max \{ \varepsilon^{p+2}, \varepsilon^{3p-2n- } \}.
\end{align}
and Theorem \ref{thm Main} is proven. 
\end{proof}

\begin{rmk}

The result can be extended to nonlinear generalized Schr\"odinger equations with linear combinations of power of the Laplacian by considering the scaling in the lowest order of the Laplacian and treating any higher order terms as error terms (which can be handled by Proposition \ref{prop Stab}). The local theory remains unaffected because the semigroup generated by the linear combination operators remains unitary. As a result, all the analysis can proceed in a similar manner.
\end{rmk}

\bibliographystyle{plain}
\bibliography{inverse}

\begin{thebibliography}{10}

\bibitem{Carles_2012}
R{\'e}mi C. and Emmanuel M.
\newblock Higher order schr{\"o}dinger equations.
\newblock {\em Journal of Physics A: Mathematical and Theoretical},
  45(39):395304, sep 2012.

\bibitem{Carles_2010}
R.~Carles, E.~Dumas, and C.~Sparber.
\newblock Multiphase weakly nonlinear geometric optics for {S}chr\"{o}dinger
  equations.
\newblock {\em SIAM J. Math. Anal.}, 42(1):489--518, 2010.

\bibitem{Carles_2008}
R.~Carles and I.~Gallagher.
\newblock Analyticity of the scattering operator for semilinear dispersive
  equations.
\newblock {\em Comm. Math. Phys.}, 286(3):1181--1209, 2009.

\bibitem{Deans}
S.~R. Deans.
\newblock {\em The Radon Transform and Some of Its Applications}.
\newblock Dover Books on Mathematics Series. Dover Publications, 2007.

\bibitem{HMG}
C.~C. Hogan, J.~Murphy, and D.~Grow.
\newblock Recovery of a cubic nonlinearity for the nonlinear {S}chr\"{o}dinger
  equation.
\newblock {\em J. Math. Anal. Appl.}, 522(1):Paper No. 127016, 9, 2023.

\bibitem{KRMPV}
R.~Killip, J.~Murphy, and M.~Visan.
\newblock The scattering map determines the nonlinearity.
\newblock {\em Proc. Amer. Math. Soc.}, 151(6):2543--2557, 2023.

\bibitem{Mur}
J.~Murphy.
\newblock Recovery of a spatially-dependent coefficient from the {NLS}
  scattering map.
\newblock {\em arXiv preprint arXiv:2209.07680}, 2022.

\bibitem{BS2}
A.~S\'{a}~Barreto and P.~Stefanov.
\newblock Recovery of a general nonlinearity in the semilinear wave equation.
\newblock {\em arXiv preprint arXiv:2107.08513}, 2021.

\bibitem{BS1}
A.~S\'{a}~Barreto and P.~Stefanov.
\newblock Recovery of a cubic non-linearity in the wave equation in the weakly
  non-linear regime.
\newblock {\em Comm. Math. Phys.}, 392(1):25--53, 2022.

\bibitem{Wa}
M.~Watanabe.
\newblock Time-dependent method for non-linear {S}chr\"{o}dinger equations in
  inverse scattering problems.
\newblock {\em J. Math. Anal. Appl.}, 459(2):932--944, 2018.

\end{thebibliography}

\end{document}